\documentclass[11pt]{elsarticle}
\usepackage{amsmath,amssymb,amsthm}
\usepackage{hyperref}
\usepackage{mathrsfs}
\usepackage{graphicx}
\usepackage{tikz}
\usepackage{enumerate}
\usepackage{float}
\usepackage{url}
\usepackage{comment}

\theoremstyle{plain}
\newtheorem{theorem}{Theorem}[section]
\newtheorem{lemma}[theorem]{Lemma}

\newtheorem{proposition}[theorem]{Proposition}

\theoremstyle{definition}
\newtheorem{definition}[theorem]{Definition}
\newtheorem{remark}[theorem]{Remark}

\newcommand{\dda}{\mathord{\mbox{\makebox[0pt][l]{\raisebox{-.4ex}{$\downarrow$}}$\downarrow$}}}

\newcommand{\ua}{\mathord{\uparrow}}
\newcommand{\da}{\mathord{\downarrow}}
\newcommand{\rom}[1]{\rm{\uppercase\expandafter{\romannumeral #1}}}

\makeatletter
\def\ps@pprintTitle{%
  \let\@oddhead\@empty
  \let\@evenhead\@empty
  \def\@oddfoot{\reset@font\hfil\thepage\hfil}
  \let\@evenfoot\@oddfoot
}
\makeatother

\begin{document}

\begin{frontmatter}

\title{A Universal 	Quasi-Polish Space}

\author[1]{Su Gao}
\ead{sgao@nankai.edu.cn}
\address[1]{School of Mathematical Sciences and LPMC, Nankai University, Tianjin 300071, China}
\author[2]{Qingguo Li}
\ead{liqingguoli@aliyun.com}
\address[2]{School of Mathematics, Hunan University, Changsha, Hunan, 410082, China}
\author[3]{Hualin Miao\corref{a1}}
\ead{miaohualinmiao@163.com}
\address[3]{College of Mathematical Sciences, Harbin Engineering University, Harbin, Heilongjiang, 150001, China}
\cortext[a1]{Corresponding author.}
\begin{abstract}
In this paper, we first construct a universal quasi-Polish space $\mathbb{Y}$ with the property that every quasi-Polish space is homeomorphic to a closed subspace of $\mathbb{Y}$. This gives a hyperspace $F(\mathbb{Y})$ of all quasi-Polish spaces which is itself a quasi-Polish space. We then prove that the collection of all Polish closed subspaces of $\mathbb{Y}$ forms a coanalytic, non-Borel subset of $F(\mathbb{Y})$. In particular, this collection is ${\mathbf{\Pi}}^1_1$-complete.
\end{abstract}


\end{frontmatter}
\section{Introduction}
Universal spaces play a fundamental role in topology and descriptive set theory. A universal space provides a single ambient space in which every member of a given class can be represented as a subspace of a prescribed type. Such representations not only simplify the study of the class itself, but also make it possible to investigate the descriptive complexity of collections of subspaces through the associated hyperspaces.

For Polish spaces, universal representations have been extensively studied. It is well-known that every Polish space is homeomorphic to a closed subspace of the countable power $\mathbb{R}^{\omega}$ (see, e.g. \cite[Theorem 4.17]{Kechris}), and on the other hand, every Polish space can also be embedded as a $G_{\delta}$-subspace of the Hilbert cube $[0,1]^{\omega}$ (see \cite[Theorem 4.14]{Kechris}). Since $\mathbb{R}^{\omega}$ is known to be homeomorphic to the Hilbert space $\ell_2$ (see \cite[Theorem VI.2.1]{BessagaPelczynski}) and the universal Urysohn metric space $\mathbb{U}$ \cite{Uspenskij}, each of these other classical spaces also represent all Polish spaces as closed subspaces. They have become indispensable tools in descriptive set theory and infinite-dimensional topology.

Quasi-Polish spaces, introduced by de Brecht \cite{B}, form a broad class of countably based spaces that extends the class of Polish spaces while retaining many desirable descriptive-set-theoretic properties. They provide a natural topological framework for domain theory, computable analysis, and generalized descriptive set theory, and have become an important setting in which classical results on Polish spaces can be extended beyond the Hausdorff context.

The study of universality is considerably less complete for classes of quasi-Polish spaces. De Brecht \cite{B} proved that every quasi-Polish space is homeomorphic to a $\mathbf{\Pi}^0_2$-subspace of the countable power $\mathbb{S}^{\omega}$ of the Sierpiński space $\mathbb{S}$, thereby providing a universal representation analogous to the $G_{\delta}$-representation of Polish spaces. However, a universal quasi-Polish space in which every quasi-Polish space can be realized as a closed subspace was not available.

The difficulty stems from the fact that quasi-Polish spaces are substantially more general than Polish spaces. In particular, quasi-Polish spaces need not be metrizable and are generally non-Hausdorff. Consequently, many techniques used in the construction of universal Polish spaces, which rely heavily on metric or Hausdorff properties, cannot be directly adapted to the quasi-Polish setting. Constructing a universal quasi-Polish space with respect to closed embeddings 
therefore requires new ideas.

The first objective of this paper is to fill this gap. We construct a universal quasi-Polish space $\mathbb{Y}$ with the property that every quasi-Polish space is homeomorphic to a closed subspace of $\mathbb{Y}$. This has a nice consequence that the hyperspace $F(\mathbb{Y})$ of all closed subsets of $\mathbb{Y}$, equipped with the lower Vietoris topology, is still a quasi-Polish space, and now it becomes the quasi-Polish hyperspace of all quasi-Polish spaces. In contrast, no such canonical Polish hyperspaces of all Polish spaces are known. For example, $F(\mathbb{R}^\omega)$, equipped with the Effros Borel structure, is only a standard Borel space and, while it carries many Polish topologies compatible with its Borel structure, none of these Polish topologies are canonical. 

Once the universal quasi-Polish space $\mathbb{Y}$ has been established, it is natural to investigate the descriptive complexity of distinguished classes of closed subspaces of $\mathbb{Y}$ as subsets of $F(\mathbb{Y})$. In particular, one may ask how complicated it is to recognize all closed subspaces of $\mathbb{Y}$ that are Polish. Answering this question, we show that the set $\{F\in F(\mathbb{Y})\colon \mbox{$F$ is Polish}\}$ is a coanalytic non-Borel subset of $F(\mathbb{Y})$. More precisely, we prove that it is ${\mathbf{\Pi}}^1_1$-complete.

In summary, we provide a closed-universal representation for all quasi-Polish spaces and determine the exact descriptive complexity of all Polish closed subspaces in our universal quasi-Polish space. These results suggest that the class of all quasi-Polish spaces is in some sense more canonical than the class of all Polish spaces.

The remainder of the paper is organized as follows. In Section 2, we recall the necessary background on quasi-Polish spaces. Section 3 is devoted to the construction of the universal quasi-Polish space $\mathbb{Y}$. In Section 4, we establish the descriptive complexity of the family of Polish closed subspaces and prove its ${\mathbf{\Pi}}^1_1$-completeness.

\section{Preliminaries}
In this section we review the basic concepts and notation about quasi-Polish spaces~\cite{B, BKS}, descriptive set theory~\cite{Kechris}, and more generally non-Hausdorff topology~\cite{clad3, nht2}. These will be used throughout the paper.

\begin{definition}
	A \emph{quasi-metric} on a set $X$ is a function $d : X \times  X \rightarrow  [0,\infty )$ such that for all $x, y, z \in  X$:
	\begin{enumerate}
	\vspace{-6pt}
	\item $x = y\iff d(x, y) = d(y, x) = 0$,
	\vspace{-6pt}
	\item $d(x, z)\leq d(x, y) +d(y, z)$.
	\end{enumerate}
\vspace{-6pt}A \emph{quasi-metric space} is a pair $(X,d)$ where $d$ is a quasi-metric on $X$.
\end{definition}
A quasi-metric $d$ on $X$ induces a $T_{0}$ topology $\tau _d$ on $X$ generated by basic open balls of the form $B_{d}(x, \varepsilon) =\{y \in X : d(x, y) < \varepsilon \}$ for $x \in X$ and a real number $\varepsilon  > 0$. We let $\bar{B}_{d}(x, \varepsilon) =\{y \in X : d(x, y) \leq \varepsilon \}$.

If $(X,d)$ is a quasi-metric space, then $(X,\hat{d})$ is a metric space, where $\hat{d}$ is defined as $\hat{d}(x,y) = \max\{d(x,y),d(y,x)\}$. The metric topology induced by $\hat{d}$ will be denoted by $\tau_{\hat{d}}$.

A sequence $(x_n)_{n\in \omega}$ in a quasi-metric space $(X,d)$ is \emph{Cauchy} if for each real
number $ \varepsilon  > 0$, there exists $n_0 \in \omega $ such that $d(x_n, x_m) < \varepsilon $ for all $m\geq n\geq n_0$. $(X,d)$ is a \emph{complete} quasi-metric space if every Cauchy sequence in $X$ converges with respect to the metric topology $\tau_{\hat{d}}$.

A topological space is \emph{countably based} (or \emph{second countable}) if its topology admits a countable base. A topological space is \emph{completely metrizable} if its topology is generated by a compatible complete metric. \emph{Completely quasi-metrizable} spaces are similarly defined.

\begin{definition}
	A topological space is \emph{Polish} if it is countably based and completely metrizable; it is \emph{quasi-Polish} if it is countably based and completely quasi-metrizable.
\end{definition}

\begin{definition}\cite{B}
Let $X$ be a topological space. For each ordinal $\alpha$ $(1\leq \alpha < \omega_{1} )$, we define $\mathbf{\Sigma}^{0}_{\alpha} (X)$ inductively as follows.
\begin{enumerate}
\item[(1)] $\mathbf{\Sigma}^{0}_{1}(X)$ is the set of all open subsets of $X$.
\item[(2)] For $\alpha > 1, \mathbf{\Sigma}^{0}_{\alpha} (X)$ is the set of all subsets $A$ of $X$ which can be expressed in the form
\begin{center}
	$A=\bigcup_{i\in \mathbb{N}}B_{i}\backslash B_{i}^{'}$,
\end{center}
where for each $i$, $B_{i}$ and $B_{i}^{'}$ are in $\mathbf{\Sigma}_{\beta_{i}}^{0}$ for some $\beta_{i}<\alpha$.
\end{enumerate}
We define $\mathbf{\Pi}^{0}_{\alpha}(X)=\{X \backslash A: A \in \mathbf{\Sigma}^{0}_{\alpha}(X)\}$ and $\mathbf{\Delta}^{0}_{\alpha}(X)=\mathbf{\Sigma}^{0 }_{\alpha} (X)\cap \mathbf{\Pi}^{ 0}_{\alpha} (X)$.
\end{definition}

Note that for Polish spaces, the above definition coincides with the classical definition of the Borel hierarchy~\cite[Propositions 6, 7]{B}. 

Let $\mathbb{S}$ be the \emph{Sierpi\'{n}ski space} on the set $\{0,1\}$ with the topology $\big\{\varnothing, \{1\}, \{0,1\}\big\}$. The product space $\mathbb{S}^\omega$ is quasi-Polish and has the universality property that a topological space is quasi-Polish if and only if it is homeomorphic to a $\mathbf{\Pi}^0_2$-subspace of $\mathbb{S}^\omega$ \cite[Theorem 24]{B}.

\begin{definition}
A \emph{Borel space} is a pair $(X,\mathcal{S})$, where $X$ is a set and $\mathcal{S}$ is 
a $\sigma$-algebra on $X$. The members of $\mathcal{S}$ are called \emph{Borel sets}. A Borel space $(X,\mathcal{S})$ is \emph{standard} if there is a Polish topology $\tau$ on $X$ with $\mathcal{S} = \mathbf{B}(\tau)$, where $\mathbf{B}(\tau)$ denotes the smallest $\sigma$-algebra generated by $\tau$.
\end{definition}

If $X$ and $Y$ are standard Borel spaces, then a function $f\colon X\to Y$ is \emph{Borel} if for every Borel subset $B\subseteq Y$, $f^{-1}(B)$ is a Borel subset of $A$. $X$ and $Y$ are \emph{Borel isomorphic} if there is a Borel bijection between them. It is well-known that all uncountable standard Borel spaces are Borel isomorphic (see, e.g., \cite[Theorem 15.6]{Kechris}). Thus there is essentially only one uncountable standard Borel space. If $\tau$ is a quasi-Polish topology on $X$, then $(X, \mathbf{B}(\tau))$ is also a standard Borel space~\cite[Theorem 14]{B}. 

\begin{definition}
If $X$ is a standard Borel space and $A\subseteq X$, we say that $A$ is \emph{analytic} if there is a standard Borel space $Y$ and a Borel function $f\colon Y\to X$ such that $f(Y)=A$. We say that $A$ is \emph{coanalytic} if its complement $A^c$ is analytic. We denote by $\mathbf{\Sigma}_1^1(X)$ the class of all analytic subsets of X, and ${\mathbf{\Pi}}^1_1(X)$ the class of all coanalytic subsets of $X$.
\end{definition}

If $X$ is a standard Borel space, then $\mathbf{\Sigma}^1_1(X)$ can be equivalently described as 
$$\mathbf{\Sigma}_1^1(X)=\big\{A \subseteq X : \mathrm{for ~some ~standard ~Borel ~space}~ Y \mathrm{and ~Borel} ~B \subseteq X \times Y, A=\mbox{proj}_X(B)\big\}.$$
By Suslin's theorem (see, e.g., \cite[Theorem 14.11]{Kechris}), a subset of a standard Borel space is Borel if and only if it is both analytic and coanalytic. Since every quasi-Polish space is also a standard Borel space, it make sense to consider analytic and coanalytic subsets of quasi-Polish spaces, and Suslin's theorem still holds for quasi-Polish spaces.

\begin{definition} 
Let $X,Y$ be standard Borel spaces and $A \subseteq X,B \subseteq Y$. A \emph{Borel reduction} of $A$ to $B$ is a Borel function $f:X \rightarrow Y$ with $f^{-1}(B) = A$, i.e., for any $x\in X$, $x \in A\iff f(x) \in B$. If $B$ is coanalytic, then we say that $B$ is \emph{$\mathbf{\Pi}^1_1$-complete} if for any coanalytic $A$, there is a Borel reduction of $A$ to $B$.
\end{definition}

It follows from a theorem of Suslin that a $\mathbf{\Pi}^1_1$-complete set in an uncountable standard Borel space is not Borel (see, e.g., \cite[Theorem 14.2]{Kechris}). A prototypical $\mathbf{\Pi}^1_1$-complete set is the set of all well-founded trees on $\omega$. We recall its definition below.

Let $\omega^{<\omega}$ be the set of all finite sequences on $\omega$. For $s\in\omega^{<\omega}$, let $|s|$ be the \emph{length} of $s$. The unique sequence of length $0$ is called the \emph{empty sequence} and is denoted $()$ or $\varnothing$. By viewing $s$ as a function from $\{0, 1,\dots,|s|-1\}$ to $\omega$, we can write $s$ as $(s(0), s(1),\dots, s(|s|-1))$. If $k\leq |s|$, then $s|_k$ is the sequence $(s(0), s(1), \dots, s(k-1))$. For each nonempty $s\in \omega^{<\omega}$, let $s^{-}=s|_{|s|-1}$ and call it the \emph{immediate predecessor} of $s$. If $s, t\in \omega^{<\omega}$, then we write $s\subseteq t$ if $s=t|_{|s|}$. In this case $s$ is called an \emph{initial segment} or a \emph{prefix} of $t$. A \emph{tree} on $\omega$ is a nonempty set $T\subseteq \omega^{<\omega}$ closed under taking initial segments. The set of all trees on $\omega$ is denoted $\mbox{TR}$. The standard topology of $\mbox{TR}$ is the subspace topology when it is viewed as a subset of $\{0,1\}^{\omega^{<\omega}}$, which is equipped the product topology, where $\{0,1\}$ is given the discrete topology. Under this standard topology, $\mbox{TR}$ is a Polish space.

For an infinite sequence $x\in \omega^\omega$ and $k\in\omega$, we also define $x|_k=(x(0), \dots, x(k-1))$, where $x$ is viewed as a function from $\omega$ to $\omega$. If $x\in \omega^\omega$ and $s\in \omega^{<\omega}$, then we write $s\subseteq x$ also when $s=x|_{|s|}$. In this case, $s$ is again called an \emph{initial segement} or a \emph{prefix} of $x$. If $T$ is a tree on $\omega$ and $x\in\omega^\omega$, then $x$ is a \emph{branch} of $T$ if for all $n\in\omega$, $x|_n\in T$. The set of all branches of $T$ is called the \emph{body} of $T$ and is denoted $[T]$. The tree $T$ is called \emph{well-founded} if $[T]=\varnothing$. The set of all well-founded trees on $\omega$ is denoted $\mbox{WF}$. $\mbox{WF}$ is a $\mathbf{\Pi}^1_1$-complete subset of $\mbox{TR}$.


In our proofs we will also use the following concepts and notation from domain theory~\cite{clad3, nht2}. 

Given a $T_{0}$ space $(X,\tau)$, for any subset $A\subseteq X$, we denote the closure of $A$ by $cl(A)$ and the complement of $A$ by $A^{c}$. We write $F(X)$ for the family of all closed subsets of $X$ and $\mathcal{O}(X)$ for the family of all open subsets of $X$.

Let $P$ be a poset, 
and $F\subseteq_{fin}P$ denote that $F$ is a finite subset of $P$. A subset $D$ of $P$ is \emph{directed} provided that it is nonempty and every finite subset of $D$ has an upper bound in $D$. 
A poset $P$ is a \emph{dcpo} (standing for \emph{directed complete poset}) if every directed subset $D$ has a supremum. 

Let $P$ be a poset and $x,y \in P$. Write $\ua x=\{z\in P\colon x\leq z\}$ and $\da x=\{z\in P\colon z\leq x\}$. We say that $x$ is \emph{way-below} $y$, denoted by $x \ll y$, if for any directed subset $D$ of $P$ for which $\sup D$ exists, $y \leq \sup D$ implies $D \cap \ua x \neq \varnothing$. The set of all elements that are way below $x$ is denoted $\dda x=\{y\in P\colon y\ll x\}$. We say that $P$ is \emph{continuous} if $\dda x$ is directed and $x=\sup \dda x$. A continuous dcpo is called a \emph{domain}. The element $x$ is called \emph{compact} iff $x\ll x$. The subset of all compact elements of $P$ is denoted $K(P)$. The poset $P$ is called \emph{algebraic} if for all $x\in P$, $\da x\cap K(P)$ is directed and $x=\sup(\da x\cap K(P))$. If $K(P)$ is a countable set, then $P$ is called \emph{$\omega$-algebraic} or \emph{countably based algebraic}.

 Let $\mathcal{P}(\omega)$ be the power set of $\omega$ ordered by subset inclusion. Then $\mathcal{P}(\omega)$ is an $\omega$-algebraic domain. 
 
Let $P$ be a poset. A subset $U$ of $P$ is \emph{Scott open} if 
\begin{enumerate}
\vspace{-6pt}\item[(i)] $U = \ua U = \{x \in P : u \leq x ~\mathrm{for ~some} ~u \in U\}$, and 
\vspace{-6pt}\item[(ii)] for every directed subset $D$ of $P$ with $\sup D$ existing and $\sup D\in U$, we have $D \cap U \neq \varnothing$.
\end{enumerate}
\vspace{-6pt}The complement of a Scott open set is called \emph{Scott closed}. Let $\sigma (P)$ denote the set of all \emph{Scott open} sets and $\Gamma(P)$ the set of all \emph{Scott closed} sets. The space $(P,\sigma(P))$ is called the \emph{Scott space} of $P$ and is written as $\Sigma P$. 

It is easy to see that the Sierpi\'{n}ski space $\mathbb{S}$ is just $\Sigma 2$, where $2=\{0,1\}$ is the natural poset with the order $0<1$. The Scott space of $\mathcal{P}(\omega)$ is homeomorphic to $\mathbb{S}^\omega=(\Sigma 2)^\omega$~\cite[Theorem I-2.11]{clad3}. It is customary to denote the Scott space of $\mathcal{P}(\omega)$ still as $\mathcal{P}(\omega)$ (instead of $\Sigma\mathcal{P}(\omega)$), so we will follow this convention. In other words, in this paper, $\mathcal{P}(\omega)$ is always equipped with the Scott topology.

For a poset $P$, we call the topology generated by basic open sets of the form $P\backslash \ua x$ for $x\in P$ \emph{the lower topology}, and denote it by $\omega(P)$. If $P$ is a dcpo, the common refinement $\sigma(P) \lor\omega(P)$ of the Scott topology and the lower topology is called the \emph{Lawson topology} and is denoted by $\lambda(P)$. The space $(P,\lambda(P))$ is written $\Lambda(P)$.

In our proof in Section~4 we will use the following lemma.

\begin{lemma}\label{lower}
Let $L$ be an $\omega$-algebraic domain. Then the elements of the closed base of the lower topology on $L$ generated by $\{\ua x:x\in L\}$ are $G_{\delta}$ sets in the Scott topology of $L$.
\end{lemma}
\begin{proof}
For each $x\in L$, since $L$ is an $\omega$-algebraic domain, we have
$
x=\sup(\da x\cap K(L)).
$
It follows that
$
\ua x=\bigcap_{k\in \da x\cap K(L)}\ua k.
$
Since $L$ is $\omega$-algebraic, $K(L)$ is countable; hence, $\da x\cap K(L)$ is countable. Moreover, by \cite[Proposition II-1.6]{clad3}, $\ua k$ is Scott open for each $k\in K(L)$. Therefore, $\ua x$ is a countable intersection of Scott-open sets; thus it is a $G_{\delta}$ set in the Scott topology of $L$.

The closed base of the lower topology generated by $\{\ua x:x\in L\}$ consists of finite unions of such sets. Since finite unions of $G_{\delta}$ sets are again $G_{\delta}$, every element of this closed base is a $G_{\delta}$ set in the Scott topology of $L$.	
\end{proof}

\begin{definition}\cite{dBPS20,BKS}
Let $\prec$ be a transitive relation on $\omega$. A subset $I\subseteq \omega$ is a \emph{rounded ideal} (with respect to $\prec$) if the following hold:
\begin{enumerate}
	\item $I\neq \varnothing$;
	\item $y\in I\land x\prec y\Longrightarrow x\in I$;
	\item $x,y\in I\Longrightarrow \exists z\in I, x\prec z\land y\prec z$.
\end{enumerate}
The collection $\mbox{RI}(\prec)$ of all rounded ideals has the topology generated by the basic open sets of the form $[n]_{\prec}=\{I\in \mbox{RI}(\prec): n\in I\}$ for $n\in \omega$. This topology coincides with the subspace topology of $\mbox{RI}(\prec)$ inherited from $\mathcal{P}(\omega)$. 

For each $n\in \omega$, we define $\da_{\prec}n=\{m\in \omega:m\prec n\}$ and $\ua_{\prec}n=\{m\in \omega:n\prec m\}$. This definition can be easily replicated if $n$ is replaced by a subset of $\omega$.
\end{definition}

By \cite[Theorem 12]{dBPS20}, all quasi-Polish spaces can be represented as rounded ideal spaces: A topological space is quasi-Polish if and only if it is homeomorphic to $\mbox{RI}(\prec)$ for some transitive relation $\prec$ on $\omega$ (also see \cite[Section 2]{BKS}).

\section{Construction of a universal quasi-Polish space}
In this section, we construct a universal quasi-Polish space $\mathbb{Y}$ with the property that every quasi-Polish space is homeomorphic to a closed subspace of $\mathbb{Y}$. 

Throughout this section, for notational simplicity, we let $2^{\omega}$ denote $\mathbb{S}^\omega=(\Sigma 2)^\omega$, which is essentially the same as $\mathcal{P}(\omega)$ with the inclusion order and the Scott topology.

Let 
$$X=(2^{\omega}\times \mathcal{O}(2^{\omega}))\times (2^{\omega}\times F(2^{\omega})), $$
where both $\mathcal{O}(2^{\omega})$ and $F(2^{\omega})$ are equipped with the Lawson topology induced by the inclusion order. Let $\mathbb{X}$ be the countable product of $X$, that is, $\mathbb{X}=X^{\omega}$. In particular, $\mathbb{X}$ has the product topology. Finally, define 
$$\mathbb{Y}=\big\{(x_n,U_n,x_n',A_n)_{n\in \omega}\in \mathbb{X}: \forall n\in \omega, x_n=x_n';\forall n\in \omega, x_n\in U_n \lor x_n'\in A_n;\forall n,m\in \omega, x_n=x_m\big\}. $$

We show that $\mathbb{Y}$ is a quasi-Polish space. Before doing so, we first study the auxiliary space $\mathbb{X}$. 
\begin{lemma}\label{Z}
The space $\mathbb{X}$ is a quasi-Polish space.
\end{lemma}
\begin{proof}
By \cite[Lemma 11]{ZhaoDongsheng}, the lattice of Scott closed subsets of an algebraic domain is algebraic. Hence, $F(2^{\omega})$ is an algebraic lattice. Since $2^{\omega}$ is countably based, the argument in the proof of \cite[Lemma 11]{ZhaoDongsheng} further implies that $F(2^{\omega})$ is a countably based algebraic domain. Therefore, $F(2^{\omega})$, equipped with the Lawson topology, is Polish by \cite[Proposition V-5.17]{clad3}.

By \cite[Exercise 8.3.22]{nht2}, $\mathcal{O}(2^{\omega})$ is an $\omega$-algebraic lattice. Hence, applying \cite[Proposition V-5.17]{clad3} again, we conclude that $\mathcal{O}(2^{\omega})$ endowed with the Lawson topology is Polish.

Moreover, $2^{\omega}$ is quasi-Polish (this is in particular a consequence of \cite[Theorem 44]{B}). Since quasi-Polish spaces are closed under countable products by \cite[Corollary 43]{B}, it follows that $\mathbb{X}$ is quasi-Polish.
\end{proof}

\begin{lemma}\label{scott}
The Scott topology of the product poset $2^{\omega}\times(\mathcal{O}(2^{\omega}),\subseteq)\times 2^{\omega}\times (F(2^{\omega}),\subseteq)$ is contained in the product topology of the space $X$.
\end{lemma}
\begin{proof}
By \cite[Theorem II-4.13]{clad3}, the Scott topology on
$
2^{\omega}\times(\mathcal{O}(2^{\omega}),\subseteq)\times 2^{\omega}\times(F(2^{\omega}),\subseteq)
$
coincides with the product topology of
$
2^{\omega}\times\Sigma(\mathcal{O}(2^{\omega}),\subseteq)\times 2^{\omega}\times\Sigma(F(2^{\omega}),\subseteq).
$
By the definition of the Lawson topology, the Scott topologies of
\(\mathcal{O}(2^{\omega})\) and \(F(2^{\omega})\) are contained in their respective Lawson topologies. Therefore, the Scott topology of the above product poset is contained in the product topology of \(X\), which proves the result.
\end{proof}

\begin{proposition}
The space $\mathbb{Y}$ is a quasi-Polish space.
\end{proposition}
\begin{proof}
According to Lemma \ref{Z}, \(\mathbb{X}\) is quasi-Polish. By \cite[Theorem 23]{B}, a subspace of a quasi-Polish space is quasi-Polish if and only if it is a \(\mathbf{\Pi}^{0}_{2}\)-subset. Therefore, it suffices to show that \(\mathbb{Y}\) is a \(\mathbf{\Pi}^{0}_{2}\)-subset of \(\mathbb{X}\). Since \(2^{\omega}\) is countably based, let \(\{B_n:n\in\omega\}\) be a countable topological base for \(2^{\omega}\).
	
We analyze the complexity of \(\mathbb{Y}\) component by component.

\(\bullet\) \(\forall n\in\omega,\ x_n=x_n'\).

 Define
$
Y_1=\big\{(x_n,U_n,x_n',F_n)_{n\in\omega}\in\mathbb{X}:\forall n\in\omega,\ x_n=x_n'\big\}.
$
Since \(2^{\omega}\) is a \(T_0\) space, we have
\[
\forall n\in\omega,\ x_n=x_n'
\mbox{ if and only if }
\forall n\in\omega,\forall k\in\omega,\ x_n\in B_k\iff x_n'\in B_k .
\]
It follows that
\[
\begin{aligned}
	Y_1
	=
	\bigcap_{n\in\omega}\bigcap_{k\in\omega}
	\Big(
	&P_n^{-1}\big(B_k^c\times\mathcal{O}(2^{\omega})\times B_k^c\times F(2^{\omega})\big)\\
	&\cup
	P_n^{-1}\big(B_k\times\mathcal{O}(2^{\omega})\times B_k\times F(2^{\omega})\big)
	\Big),
\end{aligned}
\]
where \(P_n:\mathbb{X}\rightarrow X\) denotes the \(n\)-th projection map.

\(\bullet\) \(\forall n\in\omega,\ x_n\in U_n\) or \(x_n'\in F_n\). 

Define
$
Y_2=\big\{(x_n,U_n,x_n',F_n)_{n\in\omega}\in\mathbb{X}\colon 
\forall n\in\omega,\ x_n\in U_n\lor x_n'\in F_n\big\}.
$

Let
$
G=\{(x,U,x',F)\in X\colon x\in U\}.
$
It is straightforward to verify that \(G\) is Scott open in \(X\). By Lemma \ref{scott}, \(G\) is open in \(X\).

Next, note that \(F_n\) is closed in \(2^{\omega}\) for every \(n\in\omega\). Hence,
\[
x_n'\in F_n
\mbox{ if and only if }
x_n'\in cl(F_n).
\]
Thus
\[
x_n'\in F_n
\mbox{ if and only if }
\forall k\in\omega,\ x_n'\in B_k\Longrightarrow B_k\cap F_n\neq\varnothing .
\]
Consequently, the set
$
H=\{(x,U,x',F)\in Z:x'\in F\}
$
can be written as
\[
\bigcap_{k\in\omega}
\big(
(2^{\omega}\times\mathcal{O}(2^{\omega})\times B_k^c\times F(2^{\omega}))
\cup
(2^{\omega}\times\mathcal{O}(2^{\omega})\times B_k\times\diamondsuit B_k)
\big),
\]
where
$
\diamondsuit B_k=\{F\in F(2^{\omega}):F\cap B_k\neq\varnothing\}.
$
It is immediate that \(\diamondsuit B_k\) is Scott open in
\((F(2^{\omega}),\subseteq)\), and hence, it is Lawson open.

Therefore,
\[
Y_2=\bigcap_{n\in\omega}P_n^{-1}(G\cup H),
\]
where \(P_n:\mathbb{X}\rightarrow X\) is the \(n\)-th projection map.

\(\bullet\) \(\forall n,m\in\omega,\ x_n=x_m\). 

Define
$
Y_3=\{(x_n,U_n,x_n',F_n)_{n\in\omega}\in\mathbb{Z}:
\forall n,m\in\omega,\ x_n=x_m\}.
$
Similarly to the first clause, we have
\[
\forall n,m\in\omega,\ x_n=x_m
\mbox{ if and only if }
\forall n,m\in\omega,\forall k\in\omega,\ 
x_n\in B_k\iff x_m\in B_k .
\]
Accordingly,
\[
\begin{aligned}
	Y_3
	=
	\bigcap_{(n,m)\in\omega\times\omega}
	\bigcap_{k\in\omega}
	\bigg(
	&
	\Big(
	P_n^{-1}\big(B_k^c\times\mathcal{O}(2^{\omega})\times2^{\omega}\times F(2^{\omega})\big)
	 \cap
	P_m^{-1}\big(B_k^c\times\mathcal{O}(2^{\omega})\times2^{\omega}\times F(2^{\omega})\big)
	\Big)
	\\
	&\cup
	\Big(
	P_n^{-1}\big(B_k\times\mathcal{O}(2^{\omega})\times2^{\omega}\times F(2^{\omega})\big)
	\cap
	P_m^{-1}\big(B_k\times\mathcal{O}(2^{\omega})\times2^{\omega}\times F(2^{\omega})\big)
	\Big)
	\bigg).
\end{aligned}
\]

In conclusion,
$
\mathbb{Y}=Y_1\cap Y_2\cap Y_3.
$
Therefore, \(\mathbb{Y}\) is a \(\mathbf{\Pi}^{0}_{2}\)-subset of \(\mathbb{X}\).
\end{proof}

\begin{theorem}\label{closed}
	Every quasi-Polish space is homeomorphic to a closed subspace of \(\mathbb{Y}\).
\end{theorem}	
\begin{proof}
		Let \(Z\) be a quasi-Polish space. By \cite[Theorem 23]{B}, we may assume that \(Z\) is a \(\mathbf{\Pi}^0_2\)-subset of \(2^\omega\). Hence, there exist open sets \(U_n\) and closed sets \(A_n\) in \(2^\omega\) such that
		\[
		Z=\bigcap_{n\in\omega}(A_n\cup U_n).
		\]
		Define \(f:Z\rightarrow\mathbb{Y}\) by
		\[
		(f(x))_n=(x,U_n,x,A_n),\qquad \forall n\in\omega .
		\]
		It is straightforward to verify that \(f(x)\in\mathbb{Y}\). Hence, \(f\) is well defined.
		
		We prove the theorem through the following two claims.
		
		\textbf{Claim 1:} \(f\) is a topological embedding.
		
		It is immediate that \(f\) is injective. To show that \(f\) is continuous, let \(x^k\rightarrow x\) in \(Z\). For every \(n\in\omega\), we have
		\[
		f(x^k)_n=(x^k,U_n,x^k,A_n)\rightarrow (x,U_n,x,A_n)=f(x)_n.
		\]
		Therefore,
		$
		f(x^k)\rightarrow f(x)
		$
		in \(\mathbb{Y}\); whence \(f\) is continuous.
		
		It remains to show that \(f^{-1}\) is continuous. Suppose that
		$
		f(x^k)\rightarrow f(x).
		$
		Then, for every \(n\in\omega\),
		\[
		f(x^k)_n=(x^k,U_n,x^k,A_n)\rightarrow f(x)_n=(x,U_n,x,A_n).
		\]
		It follows that
		$
		x^k\rightarrow x,
		$
		which proves the continuity of \(f^{-1}\). As a result, \(f\) is a topological embedding.
		
		\textbf{Claim 2:} \(f(Z)\) is a closed subset of \(\mathbb{Y}\).
		
		Let \(y\in cl(f(Z))\). Since \(\mathbb{Y}\) is quasi-Polish, it is first-countable; in consequence, there exists a sequence \((x^k)_{k\in\omega}\subseteq X\) such that
		$
		f(x^k)\rightarrow y.
	   $
		For every \(n\in\omega\),
		\[
		f(x^k)_n=(x^k,U_n,x^k,A_n)\rightarrow y_n .
		\]
		By the definition of \(\mathbb{Y}\), we may write
		$
		y_n=(y_n',U_n',y_n',A_n')\in Z,
		$
		where
		$
		y_n'\in U_n'\ \lor\ y_n'\in A_n'
		$
		for every \(n\in\omega\), and
		$
		y_n'=y_m'
		$
		for all \(n,m\in\omega\).
		
		Moreover, the constant sequence \(U_n\) converges to \(U_n'\) in \(\mathcal{O}(2^\omega)\), and the constant sequence \(A_n\) converges to \(A_n'\) in \(F(2^\omega)\). Since both \(\mathcal{O}(2^\omega)\) and \(F(2^\omega)\) are \(T_2\) spaces, we obtain
		$
		U_n=U_n', A_n'=A_n .
		$
		Therefore,
		$
		y_n'\in U_n\mbox{ and } y_n'\in A_n
		$
		for every \(n\in\omega\). Since
		$
		Z=\bigcap_{n\in\omega}(A_n\cup U_n),
		$
		we conclude that
		$
		y_n'\in Z .
		$
		Let \(y'=y_n'\) (which is independent of \(n\)). Then, by the definition of \(f\),
	$
		f(y')=y .
		$
		Hence,
		$
		y\in f(Z),
		$
		which shows that \(f(Z)\) is closed in \(\mathbb{Y}\).
		
		Therefore, every quasi-Polish space is homeomorphic to a closed subspace of \(\mathbb{Y}\).
	\end{proof}

\section{Descriptive complexity of the class of all Polish spaces}

For any topological space \(X\), let \(F(X)\) be the set of all closed subsets of \(X\). Equip $F(X)$ with the lower Vietoris topology, where the basic open sets are of the form
\[
\diamondsuit U=\{F\in F(X):F\cap U\neq\varnothing\},
\]
where \(U\subseteq X\) is open.

It is known that if \(X\) is a quasi-Polish space, then \(F(X)\) endowed with the lower Vietoris topology is also quasi-Polish (see, e.g., \cite[Proposition 3.1 and Example 3.2]{dBK19}). Because of the universality property of the space \(\mathbb{Y}\) constructed in the preceding section, $F(\mathbb{Y})$ becomes a quasi-Polish hyperspace of all quasi-Polish spaces. In particular, \(F(\mathbb{Y})\) is a standard Borel space. 

Consider
$
\mathsf{P}=\{F\in F(\mathbb{Y}):F\ \mathrm{is\ Polish}\}.
$
In the rest of this section we investigate the descriptive complexity of \(\mathsf{P}\) as a subset of $F(\mathbb{Y})$. We first note the following result.
\begin{proposition}
$\mathsf{P}$ is a coanalytic subset of $F(\mathbb{Y})$.
\end{proposition}
\begin{proof}
From \cite[Corollary 5.3]{Ruiyuan}, a quasi-Polish space is Polish if and only if it is regular.
Since \(\mathbb{Y}\) is countably based, let \(\{U_n:n\in\omega\}\) be a countable topological base for \(\mathbb{Y}\). For every \(F\in F(\mathbb{Y})\), \(F\) is not regular if and only if there exist \(x_0\in F\) and \(n\in\omega\) such that \(x_0\in U_n\) and, for every \(k\in\omega\) with \(x_0\in U_k\subseteq U_n\), we have
$
cl(F\cap U_k)\not\subseteq U_n.
$

For any \(n,k\in\omega\), define
$
B_k^n=\big\{F\in F(\mathbb{Y}): cl(F\cap U_k)\not\subseteq U_n\big\}.
$
By the definition of \(B_k^n\), we have
\[
\begin{aligned}
	F\in B_k^n
	&\Longleftrightarrow
	\exists a\in cl(F\cap U_k)\setminus U_n\\
	&\Longleftrightarrow
	\exists a\in\mathbb{Y}\,
	\left[
	a\in U_n^c
	\ \wedge\
	\forall m\in\omega\,
	\bigl(a\in U_m\Rightarrow U_m\cap F\cap U_k\neq\varnothing\bigr)
	\right].
\end{aligned}
\]
It follows that \(B_k^n\) is analytic for every \(n,k\in\omega\). Consequently, \(F(\mathbb{Y})\setminus \mathsf{P}\) is analytic. Therefore, \(\mathsf{P}\) is a coanalytic subset of \(F(\mathbb{Y})\).
\end{proof}

Next, we prove that \(\mathsf{P}\) is not Borel. To this end, we define some isomorphic copies of $\mathbb{Y}$ and $\mathsf{P}$, and work with them. 

Let $\omega^{<\omega}$ be the set of all finite sequences in $\omega$. For each $s\in \omega^{\textless\omega}$, let $\underline{s}$ be a new symbol, and let $\Omega=\{s,\underline{s}: s \in \omega^{\textless\omega}\}$. \(\Omega\) is still countably infinite and there is a bijection between \(\Omega\) and \(\omega\). We fix once and for all such a bijection, and view $\Omega$ as a copy of $\omega$.
Let \(2^{\Omega}\) denote the product space $(\Sigma 2)^{\Omega}$. Then $2^{\Omega}$ is homeomorphic to $2^\omega$. As in the definition of $\mathbb{Y}$, equip both \(\mathcal{O}(2^{\Omega})\) and \(F(2^{\Omega})\) with their respective Lawson topologies. Let $$\overline{X}=(2^{\Omega}\times \mathcal{O}(2^{\Omega}))\times (2^{\Omega}\times F(2^{\Omega})) $$
and $\overline{\mathbb{X}}$ be the countable product of copies of $\overline{X}$, that is, $\overline{\mathbb{X}}=\overline{X}^{\omega}$. Finally, define 
$$\overline{\mathbb{Y}}=\big\{(x_n,U_n,x_n',A_n)_{n\in \omega}\in \overline{\mathbb{X}}: \forall n\in \omega, x_n=x_n';\forall n\in \omega, x_n\in U_n \lor x_n'\in A_n;\forall n,m\in \omega, x_n=x_m\big\}. $$
It is easily seen that \(\overline{\mathbb{Y}}\) is homeomorphic to \(\mathbb{Y}\). In particular, \(\overline{\mathbb{Y}}\) is quasi-Polish.

Let $\overline{\mathsf{P}}=\{F\in F(\overline{\mathbb{Y}}):F~\mathrm{is ~Polish}\}$. Our main theorem of this section is the following.

\begin{theorem}\label{Pi}
$\overline{\mathsf{P}}$ is ${\mathbf{\Pi}}^1_1$-complete. In particular, $\overline{\mathsf{P}}$ is not Borel. Consequently, $\mathsf{P}$ is $\mathbf{\Pi}^1_1$-complete and not Borel.
\end{theorem}

The rest of this section is devoted to a proof of Theorem~\ref{Pi}.

Recall that the set $\mbox{WF}$ of all well-founded trees on $\omega$ is ${\mathbf{\Pi}}^1_1$-complete. Therefore, it suffices to show that there is a Borel reduction from $\mbox{WF}$ to $\overline{\mathsf{P}}$. More precisely, we will define a Borel function $f\colon \mbox{TR}\to F(\overline{\mathbb{Y}})$ so that $\mbox{WF}=f^{-1}(\overline{\mathsf{P}})$.

To define $f$, we use some construction from \cite[Section 8]{BKS}. For each $T\in \mbox{TR}$, we define a transitive binary relation $\prec_T$ on $\Omega$ as follows. For every nonempty $s\in \omega^{<\omega}$, define
\begin{enumerate}
\vspace{-4pt}	\item $s^{-}\prec_T^1 s$;
\vspace{-4pt}	\item $\underline{s^{-}}\prec_T^2 \underline{s}$;
\vspace{-4pt}	\item $s\prec_T^3 \underline{s}$;
\vspace{-4pt}	\item $\underline{s^{-}}\prec_T^4 s$, whenever $s\notin T$.
\end{enumerate}
\vspace{-4pt}The relation $\prec_T$ is defined as the transitive closure of the four relations $\prec_T^1, \prec_T^2, \prec_T^3, \prec_T^4$.

It was shown in the proof of \cite[Theorem 8]{BKS} that, for each tree $T\in\mbox{TR}$, 
$$\mbox{RI}(\prec_T)=\big\{I_x:x\in \omega^{\omega}\big\}\cup \big\{J_x:x\in [T]\big\},$$ 
where $I_x=\{s,\underline{s}: s\in\omega^{<\omega}, s\subseteq x\}$ and $J_x=\{s\in \omega^{\textless \omega}:s\subseteq x\}$. Furthermore, $\mbox{RI}(\prec_T)$ is a $\mathbf{\Pi}^0_2$-subset of $2^{\Omega}$ by the following computation:
$$\mbox{RI}(\prec_T)=\big(2^{\Omega}\backslash \{\varnothing\}\big)\cap \left( \bigcap_{\sigma \in \Omega}\bigg(U_{\sigma}^c\cup \bigcap_{\rho\in \da_{\prec_T}\sigma}U_\rho\bigg)\right)\cap \left(\bigcap_{\sigma,\tau\in \Omega \times \Omega}\bigg( (U_\sigma\cap U_{\tau})^c\cup \bigcup_{\rho\in \ua_{\prec_T}\sigma \cap \ua_{\prec_T}\tau}U_\rho\bigg)\right),$$
where $U_{\sigma}=\{I\in 2^{\Omega}: \sigma \in I\}$ for each $\sigma\in \Omega$. It is immediate that $U_{\sigma}$ is an open subset of $2^{\Omega}$.

Let $O\subseteq \omega$ be the set of all odd numbers and $E\subseteq \omega$ be the set of all positive even numbers. Fix two bijections $h_1,h_2$ from $\Omega \times \Omega$ to $O$ and $E$, respectively. For each $T\in \mbox{TR}$, define an open subset $U_n^T$ and a closed subset $A_n^T$ of $2^{\Omega}$ for each $n\in \omega$ as follows.
\begin{enumerate}
	\item For $n=0$, let $U_0^T=2^{\Omega}\backslash \{\varnothing\}$ and $A_0^T=\varnothing$;
	\item For $n\in O$, let $(\sigma,\tau)\in \Omega \times \Omega$ be such that $n=h_1(\sigma,\tau)$. If $\sigma\prec_T \tau$, let $U_n^T=U_\sigma$ and $A_n^T=U_{\tau}^c$. Otherwise, let $U_n^T=2^{\Omega}$ and $A_n^T=2^{\Omega}$;
	\item For $n\in E$, let $(\sigma,\tau)\in \Omega \times \Omega$ be such that $n=h_2(\sigma,\tau)$. Define $U_n^T=\bigcup_{\rho\in \ua_{\prec_T}\sigma\cap \ua_{\prec_T}\tau}U_\rho$ and $A_n^T=(U_{\sigma}\cap U_\tau)^c$.
\end{enumerate}
It is easy to see that $\mbox{RI}(\prec_T)=\bigcap_{n\in \omega}(U_n^T\cup A_n^T)$. 

We are now ready to define $f\colon \mbox{TR}\to F(\overline{\mathbb{Y}})$. For any $T\in \mbox{TR}$, we define
$$f(T)=\eta(\mbox{RI}(\prec_T)),$$
where $\eta(I)=(I,U_n^T,I,A_n^T)_{n\in \omega}$ for each $I\in \mbox{RI}(\prec_T)$. 

As in the proof of Theorem \ref{closed}, $f$ is well defined. 

By the proof of \cite[Theorem 8]{BKS}, we have that $f(T)\in \overline{\mathsf{P}}$ if and only if $T\in \mbox{WF}$. 

To complete the proof of Theorem~\ref{Pi}, it remains to verify that $f$ is Borel. For this we first prove some lemmas and a proposition.

\begin{lemma}\label{subbasis}
	 Let $F_i\subseteq_{fin}2^{\Omega}$, $\mathcal{U}_i\in \Lambda(\mathcal{O}(2^{\Omega}))$, $\mathcal{V}_i\in \Lambda(F(2^{\Omega}))$, and $A_{F_i}=\ua F_i=\{B\in 2^{\Omega}: F_i\subseteq B\}$. Let $P_i$ denote the $i$-th projection map. Let $f$ be defined above. If 
	 $$f^{-1}\left(\diamondsuit \bigg(\bigcap_{i\leq n}P_i^{-1}(A_{F_i}\times \mathcal{U}_i\times A_{F_i}\times \mathcal{V}_i)\bigg)\right)\neq\varnothing, $$then for all $n\in \omega$, 
	 $$f^{-1}\left(\diamondsuit \bigg(\bigcap_{i\leq n}P_i^{-1}(A_{F_i}\times \mathcal{U}_i\times A_{F_i}\times \mathcal{V}_i)\bigg)\right)=f^{-1}\left(\bigcap_{i\leq n}\diamondsuit \Big(P_i^{-1}(A_{F_i}\times \mathcal{U}_i\times A_{F_i}\times \mathcal{V}_i)\Big)\right) .$$
\end{lemma}	 
\begin{proof}
	 	The inclusion from left to right is immediate. 
	 	
	 	For the converse, let
	 	$
	 	T\in f^{-1}\big(\bigcap_{i\leq n}\diamondsuit P_i^{-1}(A_{F_i}\times \mathcal{U}_i\times A_{F_i}\times \mathcal{V}_i)\big).
	 	$
	 	Then
	$$
	 		f(T)
	 		=\eta(\mbox{RI}(\prec_T))
	 		\in
	 		\bigcap_{i\leq n}
	 		\diamondsuit
	 		P_i^{-1}
	 		\left(
	 		A_{F_i}\times\mathcal{U}_i\times A_{F_i}\times\mathcal{V}_i
	 		\right),
	 		$$
	 		which means for all $i\leq n$, 
	 		\[\eta(\mbox{RI}(\prec_T))
	 		\cap
	 		P_i^{-1}
	 		\left(
	 		A_{F_i}\times\mathcal{U}_i\times A_{F_i}\times\mathcal{V}_i
	 		\right)
	 		\neq\varnothing.
	 		\]
	It follows that, for each $i\leq n$, there exists some $I_i\in \mbox{RI}(\prec_T)$ such that
	 $
	 		\eta(I_i)
	 		\in
	 		P_i^{-1}(
	 		A_{F_i}\times\mathcal{U}_i\times
	 		A_{F_i}\times\mathcal{V}_i
	 		).
	 $
	 	Hence,
	 	$
	 		P_i(\eta(I_i))
	 		=(I_i,U_i^T,I_i,A_i^T)
	 		\in
	 		A_{F_i}\times\mathcal{U}_i\times
	 		A_{F_i}\times\mathcal{V}_i.
	 $
	 	By our assumption,
	 	$f^{-1}\big(\diamondsuit (\bigcap_{i\leq n}P_i^{-1}(A_{F_i}\times \mathcal{U}_i\times A_{F_i}\times \mathcal{V}_i))\big)\neq \varnothing.$
	 	Thus, we can choose
	 	$
	 	T'\in
	 	f^{-1}(\diamondsuit (\bigcap_{i\leq n}P_i^{-1}(A_{F_i}\times \mathcal{U}_i\times A_{F_i}\times \mathcal{V}_i))).
	 	$
	 	Then
	 	$
	 	\eta(\mbox{RI}(\prec_{T'}))\cap
	 	\bigcap_{i\leq n}
	 	P_i^{-1}\left(
	 	A_{F_i}\times\mathcal{U}_i\times A_{F_i}\times\mathcal{V}_i
	 	\right)
	 	\neq\varnothing.
	 	$
	 	Therefore, we can choose some $I\in \mbox{RI}(\prec_{T'})$ such that
	 	$
	 	\eta(I)\in
	 	\bigcap_{i\leq n}
	 	P_i^{-1}\left(
	 	A_{F_i}\times\mathcal{U}_i\times A_{F_i}\times\mathcal{V}_i
	 	\right).
	 	$
	 	This implies that for each $i\leq n$,
	 	$
	 	P_i(\eta(I))
	 	=(I,U_i^{T'},I,A_i^{T'})
	 	\in
	 	A_{F_i}\times\mathcal{U}_i\times A_{F_i}\times\mathcal{V}_i.
	 	$
	 	In particular, by the construction of $A_{F_i}$, we have $F_i\subseteq I$ for each $i\leq n$.
	 	
	 	By the construction of $\mbox{RI}(\prec_{T'})$, we know that
	 	$I=J_x$ or $I=I_x$
	 	for some $x\in\omega^{<\omega}$. Since $J_x\subseteq I_x$, it follows that $I_x\in\bigcap_{i\leq n}A_{F_i}$.
	 	Note that $I_x\in \mbox{RI}(\prec_T)$. Therefore, for each $i\leq n$,
	 	$
	 	P_i(\eta(I_x))
	 	=(I_x,U_i^T,I_x,A_i^T)
	 	\in
	 	P_i(\eta(\mbox{RI}(\prec_T)))
	 	=P_i(f(T)).$
	 	Moreover, since
	 	$
	 	P_i(\eta(I_i))
	 	=(I_i,U_i^T,I_i,A_i^T)
	 	\in
	 	A_{F_i}\times\mathcal{U}_i\times A_{F_i}\times\mathcal{V}_i,
	 	$
	 	we have
	 	$
	 	U_i^T\in\mathcal{U}_i$
	 	and
	 	$A_i^T\in\mathcal{V}_i
	 $
	 	for each $i\leq n$. Consequently, for all $i\leq n$, we have
	 	$$
	 	P_i(\eta(I_x))
	 	=(I_x,U_i^T,I_x,A_i^T)
	 	\in
	 	\left(
	 	A_{F_i}\times\mathcal{U}_i\times A_{F_i}\times\mathcal{V}_i
	 	\right)
	 	\cap P_i(f(T)),
	 	$$
	 	It follows that
	 	$
	 	\eta(I_x)\in
	 	\bigcap_{i\leq n}
	 	P_i^{-1}\left(
	 	A_{F_i}\times\mathcal{U}_i\times A_{F_i}\times\mathcal{V}_i
	 	\right)
	 	\cap f(T).
	 	$
	 	Hence,
	 	$
	 	f(T)\cap
	 	\bigcap_{i\leq n}
	 	P_i^{-1}(
	 	A_{F_i}\times\mathcal{U}_i\times A_{F_i}\times\mathcal{V}_i
	 	)
	 	\neq\varnothing.
	 	$
	 	Equivalently,
	 	$
	 	f(T)\in
	 	\diamondsuit(
	 	\bigcap_{i\leq n}
	 	P_i^{-1}\big(
	 	A_{F_i}\times\mathcal{U}_i\times A_{F_i}\times\mathcal{V}_i
	 	)
	 	\big).
	 	$
	 	Therefore,
	 	$
	 	T\in
	 	f^{-1}\big(\diamondsuit (\bigcap_{i\leq n}P_i^{-1}(A_{F_i}\times \mathcal{U}_i\times A_{F_i}\times \mathcal{V}_i))\big),
	 	$
	 	which yields the desired result.
	 \end{proof}

\begin{lemma}\label{cap}
 Let $F\subseteq_{fin}2^{\Omega},(\mathcal{U}_i)_{i\in \omega}\in \Lambda(\mathcal{O}(2^{\Omega})),\mathcal{V}\in \Lambda(F(2^{\Omega}))\}$, and $A_{F}=\ua F=\{B\in 2^{\Omega}: F\subseteq B\}$. Let $P_i$ denotes the $i$-th projection map.
 Let $f$ be defined above. Then for all $n\in \omega$,
 $$f^{-1}\left(\diamondsuit \bigg( P_n^{-1}\Big(A_{F}\times \bigcap_{i\in \omega}\mathcal{U}_i\times A_{F}\times \mathcal{V}\Big)\bigg)\right)=f^{-1}\left(\bigcap_{i\in \omega}\diamondsuit \big(P_n^{-1}(A_{F}\times \mathcal{U}_i\times A_{F}\times \mathcal{V})\big)\right).$$  
 \end{lemma} 
 \begin{proof}
 The inclusion from left to right is immediate. 
 
 For the converse, let
 $
 T\in f^{-1}(\bigcap_{i\in\omega}\diamondsuit P_n^{-1}(A_F\times\mathcal{U}_i\times A_F\times\mathcal{V})).
 $
 Then
 $
 f(T)=\eta(\mbox{RI}(\prec_T))\in\bigcap_{i\in\omega}\diamondsuit P_n^{-1}\left(A_F\times\mathcal{U}_i\times A_F\times\mathcal{V}\right).
 $
 Thus, for each $i\in\omega$,
 $
 \eta(\mbox{RI}(\prec_T))\in\diamondsuit P_n^{-1}\left(A_F\times\mathcal{U}_i\times A_F\times\mathcal{V}\right).
 $
 It follows that for each $i\in\omega$, there exists $I_i\in \mbox{RI}(\prec_T)$ such that
 $
 P_n(\eta(I_i))=(I_i,U_n^T,I_i,A_n^T)\in A_F\times\mathcal{U}_i\times A_F\times\mathcal{V}.
 $
 Consequently,
 $
 U_n^T\in\bigcap_{i\in\omega}\mathcal{U}_i
 $ and $
 A_n^T\in\mathcal{V}.
 $
 Taking $i=1$, we obtain
 $
 P_n(\eta(I_1))
 =(I_1,U_n^T,I_1,A_n^T)
 \in
 A_F\times\bigcap_{i\in\omega}\mathcal{U}_i\times A_F\times\mathcal{V}.
 $
 Therefore,
 $$
 \eta(I_1)
 \in
 P_n^{-1}\Big(A_F\times\bigcap_{i\in\omega}\mathcal{U}_i\times A_F\times\mathcal{V}\Big)
 \cap\eta(\mbox{RI}(\prec_T))
 $$
 and hence
 $
 f(T)\cap
 P_n^{-1}\left(A_F\times\bigcap_{i\in\omega}\mathcal{U}_i\times A_F\times\mathcal{V}\right)
 \neq\varnothing.
 $
 Thus,
 $
 f(T)\in
 \diamondsuit P_n^{-1}\left(A_F\times\bigcap_{i\in\omega}\mathcal{U}_i\times A_F\times\mathcal{V}\right).
 $
 Consequently,
 $
 T\in
 f^{-1}\left(
 \diamondsuit P_n^{-1}\left(A_F\times\bigcap_{i\in\omega}\mathcal{U}_i\times A_F\times\mathcal{V}\right)
 \right),
 $
 which proves the desired inclusion.
 \end{proof}

\begin{proposition}\label{Borel}
Let $f$ be defined above. Then $f$ is Borel if and only if 
$$f^{-1}\big(\diamondsuit P_n^{-1}(A_F\times \mathcal{U}\times A_F\times \mathcal{V})\big)$$ is a Borel set for every $F\subseteq_{fin}2^{\Omega},\mathcal{U}\in \Sigma(\mathcal{O}(2^{\Omega})),\mathcal{V}\in \Lambda(F(2^{\Omega}))$, $n\in \omega$, where $A_{F}=\ua F=\{B\in 2^{\Omega}: F\subseteq B\}$.
\end{proposition}
\begin{proof}
The implication from left to right is immediate. For the converse, let $F\subseteq_{fin}2^{\Omega},k'\in \omega,\mathcal{V}\in \Lambda(F(2^{\Omega}))$.

\textbf{Claim 1.} The set $f^{-1}(\diamondsuit P_n^{-1}(A_F\times\mathcal{U}\times A_F\times\mathcal{V}))$ is Borel for every $\mathcal{U}=\ua\{U_k:k\leq k'\}, U_k\in \mathcal{O}(2^{\Omega})$.

	
Since $\mathcal{O}(2^{\Omega})$ is an $\omega$-algebraic domain, $\mathcal{U}$ is $G_{\delta}$ in the Scott topology of $\mathcal{O}(2^{\Omega})$ by Lemma \ref{lower}. Hence, $\mathcal{U}=\bigcap_{i\in \omega}\mathcal{U}_i$, where $\mathcal{U}_i\in \Sigma(\mathcal{O}(2^{\Omega}))$ for each $i\in \omega$. By Lemma \ref{cap}, we have 
$$\begin{array}{rcl} f^{-1}(\diamondsuit P_n^{-1}(A_F\times \mathcal{U}\times A_F\times \mathcal{V}))&=&f^{-1}(\diamondsuit P_n^{-1}(A_{F}\times \bigcap_{i\in \omega}\mathcal{U}_i\times A_{F}\times \mathcal{V})) \\
&=&f^{-1}(\bigcap_{i\in \omega}\diamondsuit P_n^{-1}(A_{F}\times \mathcal{U}_i\times A_{F}\times \mathcal{V})).
\end{array}$$ By our assumption, $f^{-1}(\diamondsuit P_n^{-1}(A_{F}\times \mathcal{U}_i\times A_{F}\times \mathcal{V}))$ is Borel for each $i\in \omega$. It follows that $$f^{-1}(\bigcap_{i\in \omega}\diamondsuit P_n^{-1}(A_{F}\times \mathcal{U}_i\times A_{F}\times \mathcal{V}))=\bigcap_{i\in \omega}f^{-1}(\diamondsuit P_n^{-1}(A_{F}\times \mathcal{U}_i\times A_{F}\times \mathcal{V}))$$ is Borel. Therefore, $f^{-1}(\diamondsuit P_n^{-1}(A_F\times \mathcal{U}\times A_F\times \mathcal{V}))$ is Borel.
This proves Claim 1.

Let 
$$\mathcal{B}=\Big\{\mathcal{T}\backslash \mathcal{U}'\colon \mathcal{T}\in \sigma(\mathcal{O}(2^{\Omega})),\ \mathcal{U}'=\ua\{U_k:k\leq k'\},\ U_k\in \mathcal{O}(2^{\Omega}),\ k'\in \omega\Big\}. $$




\textbf{Claim 2}: $f^{-1}(\diamondsuit P_n^{-1}(A_F\times \mathcal{G}\times A_F\times \mathcal{V}))$ is Borel for every $\mathcal{G}\in \mathcal{B}$.

By the construction of $\mathcal{B}$, we have $\mathcal{G}=\mathcal{H}\backslash \mathcal{R}$, where $\mathcal{H}\in \sigma(\mathcal{O}(2^{\Omega}))$ and $\mathcal{R}=\ua\{U_k:k\leq k'\}$, with $U_k\in \mathcal{O}(2^{\Omega})$ for some $k'\in \omega$. It follows that
$f^{-1}(\diamondsuit P_n^{-1}(A_F\times \mathcal{G}\times A_F\times \mathcal{V}))=f^{-1}(\diamondsuit P_n^{-1}(A_F\times (\mathcal{H}\cap \mathcal{R}^c)\times A_F\times \mathcal{V}))$.
By Lemma \ref{cap}, we obtain
$$f^{-1}(\diamondsuit P_n^{-1}(A_F\times (\mathcal{H}\cap \mathcal{R}^c)\times A_F\times \mathcal{V}))=f^{-1}(\diamondsuit P_n^{-1}(A_F\times \mathcal{H}\times A_F\times \mathcal{V})\cap \diamondsuit P_n^{-1}(A_F\times \mathcal{R}^c\times A_F\times \mathcal{V}))$$
and hence
$$f^{-1}(\diamondsuit P_n^{-1}(A_F\times (\mathcal{H}\cap \mathcal{R}^c)\times A_F\times \mathcal{V}))=f^{-1}(\diamondsuit P_n^{-1}(A_F\times \mathcal{H}\times A_F\times \mathcal{V}))\cap f^{-1}(\diamondsuit P_n^{-1}(A_F\times \mathcal{R}^c\times A_F\times \mathcal{V})).$$
By Claim 1, $f^{-1}(\diamondsuit P_n^{-1}(A_F\times \mathcal{R}^c\times A_F\times \mathcal{V}))$ is Borel. Furthermore, by our assumption, $f^{-1}(\diamondsuit P_n^{-1}(A_F\times \mathcal{H}\times A_F\times \mathcal{V}))$ is Borel. Therefore,
$f^{-1}(\diamondsuit P_n^{-1}(A_F\times \mathcal{G}\times A_F\times \mathcal{V}))$ 
 is Borel. This proves Claim 2.


Next, let 
$$\mathcal{B}'=\Big\{\bigcap_{i\leq n} P_i^{-1}(A_{F_i}\times \mathcal{W}_i\times A_{F_i}\times \mathcal{V}_i):F_i\subseteq_{fin}2^{\Omega},\ \mathcal{W}_i\in \mathcal{B},\ \mathcal{V}_i\in \Lambda(F(2^{\Omega}))\Big\}. $$ 

\textbf{Claim 3}: $f^{-1}(\diamondsuit \bigcap_{i\leq n} P_i^{-1}(A_{F_i}\times \mathcal{W}_i\times A_{F_i}\times \mathcal{V}_i)))$ is Borel for each $\bigcap_{i\leq n} P_i^{-1}(A_{F_i}\times \mathcal{W}_i\times A_{F_i}\times \mathcal{V}_i)\in \mathcal{B}'$.

If $f^{-1}(\diamondsuit \bigcap_{i\leq n} P_i^{-1}(A_{F_i}\times \mathcal{W}_i\times A_{F_i}\times \mathcal{V}_i)))$ is empty, then it is Borel.

Otherwise, $f^{-1}(\diamondsuit \bigcap_{i\leq n} P_i^{-1}(A_{F_i}\times \mathcal{W}_i\times A_{F_i}\times \mathcal{V}_i)))$ is nonempty. By Lemma \ref{subbasis}, we have
$f^{-1}(\diamondsuit \bigcap_{i\leq n} P_i^{-1}(A_{F_i}\times \mathcal{W}_i\times A_{F_i}\times \mathcal{V}_i)))=f^{-1}(\bigcap_{i\leq n}\diamondsuit  P_i^{-1}(A_{F_i}\times \mathcal{W}_i\times A_{F_i}\times \mathcal{V}_i)))$.
By Claim 3, for each $i\leq n$,
$f^{-1}(\diamondsuit  P_i^{-1}(A_{F_i}\times \mathcal{W}_i\times A_{F_i}\times \mathcal{V}_i)))$
is Borel. Therefore,
$f^{-1}(\diamondsuit \bigcap_{i\leq n} P_i^{-1}(A_{F_i}\times \mathcal{W}_i\times A_{F_i}\times \mathcal{V}_i)))$
is Borel.

\textbf{Claim 4}: $f$ is Borel.

It suffices to show that $f^{-1}(\diamondsuit \mathcal{S})$ is Borel for every $\mathcal{S}\in \mathcal{O}(\overline{\mathbb{Y}})$. For each $\mathcal{S}\in \mathcal{O}(\overline{\mathbb{Y}})$, we have $\mathcal{S}=\bigcup_{i\in \omega}\mathcal{P}_i$, where $\mathcal{P}_i\in \mathcal{B}'$. It follows that
$f^{-1}(\diamondsuit \mathcal{S})=f^{-1}(\diamondsuit \bigcup_{i\in \omega}\mathcal{P}_i)=f^{-1}(\bigcup_{i\in \omega}\diamondsuit \mathcal{P}_i)$.
By Claim 3, $f^{-1}(\diamondsuit \mathcal{P}_i)$ is Borel for every $i\in \omega$. Hence, $f^{-1}(\diamondsuit \mathcal{S})$ is Borel, since it is a countable union of Borel sets.
\end{proof}

With Proposition~\ref{Borel}, we now turn to the verification of the Borelness of $f$. 
To this end, let $F$ be a finite subset of $\Omega$, let $\mathcal{U}$ be a Scott open subset of $\mathcal{O}(2^{\Omega})$, let $\mathcal{V}$ be a Lawson open subset of $F(2^{\Omega})$, and let $n\in\omega$. By Proposition \ref{Borel}, it suffices to show that $f^{-1}(\diamondsuit P_n^{-1}(A_F\times \mathcal{U}\times A_F\times \mathcal{V}))$ is Borel.  We only consider the nontrivial case where $f^{-1}(\diamondsuit P_n^{-1}(A_F\times \mathcal{U}\times A_F\times \mathcal{V}))$ is nonempty.
We distinguish the following three cases.

\textbf{Case 1}: $n=0$. We claim that $f^{-1}(\diamondsuit (P_0^{-1}(A_F\times \mathcal{U}\times A_F\times \mathcal{V})))=\mbox{TR}$ whenever it is nonempty.

Assume $f^{-1}(\diamondsuit (P_0^{-1}(A_F\times \mathcal{U}\times A_F\times \mathcal{V})))$ is nonempty. Pick $\overline{T}\in f^{-1}(\diamondsuit (P_0^{-1}(A_F\times \mathcal{U}\times A_F\times \mathcal{V})))$. Then
$P_0(f(\overline{T}))=P_0(\eta(\mbox{RI}(\prec_{\overline{T}})))\cap (A_F\times \mathcal{U}\times A_F\times \mathcal{V})\neq \varnothing$.
Hence, there exists $I\in \mbox{RI}(\prec_{\overline{T}})$ such that
\(P_0(\eta(I))=(I,U_0^{\overline{T}},I,A_0^{\overline{T}})=(I,2^{\Omega}\backslash \{\varnothing\},I,\varnothing)\in A_F\times \mathcal{U}\times A_F\times \mathcal{V}.\)

If $I=I_x$ for some $x\in \omega^{\omega}$, then $I_x\in \mbox{RI}(\prec_{T})$ for every $T\in \mbox{TR}$. Moreover,
$P_0(\eta(I_x))=(I,2^{\Omega}\backslash \{\varnothing\},I,\varnothing)\in P_0(\eta(\mbox{RI}(\prec_{T})))\cap (A_F\times \mathcal{U}\times A_F\times \mathcal{V})$ for each $T\in \mbox{TR}$.
Therefore,
$\mbox{TR}\subseteq f^{-1}(\diamondsuit (P_0^{-1}(A_F\times \mathcal{U}\times A_F\times \mathcal{V})))$.
Thus,
$f^{-1}(\diamondsuit (P_0^{-1}(A_F\times \mathcal{U}\times A_F\times \mathcal{V})))=\mbox{TR}$.

Otherwise, $I=J_x$ for some $x\in [\overline{T}]$. Then $J_x\subseteq I_x$ and $I_x\in \mbox{RI}(\prec_{T})$ for every $T\in \mbox{TR}$. Since $I=J_x\in A_F$, we have $I_x\in A_F$. Consequently,
$P_0(\eta(I_x))=(I_x,2^{\Omega}\backslash \{\varnothing\},I_x,\varnothing)\in P_0(\eta(\mbox{RI}(\prec_{T})))\cap (A_F\times \mathcal{U}\times A_F\times \mathcal{V})$ for each $T\in \mbox{TR}$.
It follows that
$f^{-1}(\diamondsuit (P_0^{-1}(A_F\times \mathcal{U}\times A_F\times \mathcal{V})))=\mbox{TR}$.

\textbf{Case 2}: $n\in O$. Then there exists $(\sigma,\tau)\in \Omega\times \Omega$ such that $h_1(\sigma,\tau)=n$. We only consider the nontrivial case where $f^{-1}(\diamondsuit P_n^{-1}(A_F\times \mathcal{U}\times A_F\times \mathcal{V}))$ is nonempty. We consider two subcases each of which is split into more sub-subcases, etc.

Case 2.1: $\exists~ \overline{T}\in f^{-1}(\diamondsuit (P_n^{-1}(A_F\times \mathcal{U}\times A_F\times \mathcal{V})))$ such that $\tau\prec_{\overline{T}}\sigma$. 

Case 2.1.1: $\prec_{\overline{T}}$ belongs to the transitive closure of $\prec_{\overline{T}}^1,\prec_{\overline{T}}^2,\prec_{\overline{T}}^3$. Then $\tau\prec_T \sigma$ for each $T\in \mbox{TR}$. Hence, $U_n^T=U_\tau$ and $A_n^T=U_{\sigma}^c$ for each $T\in TR$. We claim that $f^{-1}(\diamondsuit P_n^{-1}(A_F\times \mathcal{U}\times A_F\times \mathcal{V}))=\mbox{TR}$.

Since $\overline{T}\in f^{-1}(\diamondsuit P_n^{-1}(A_F\times \mathcal{U}\times A_F\times \mathcal{V}))$, there exists $I\in \mbox{RI}(\prec_{\overline{T}})$ such that
$P_n(\eta(I))=(I,U_n^{\overline{T}},I,A_n^{\overline{T}})=(I,U_\tau,I,U_{\sigma}^c)\in A_F\times \mathcal{U}\times A_F\times \mathcal{V}$.
Note that $I=I_x$ or $I=J_x$ for some $x\in \omega^{\omega}$. In either case, there exists $x\in \omega^{\omega}$ such that $I\subseteq I_x$ and $I_x\in \mbox{RI}(\prec_T)$ for each $T\in \mbox{TR}$. It follows that
$P_n(\eta(I_x))\in P_n(\eta(\mbox{RI}(\prec_T)))\cap (A_F\times \mathcal{U}\times A_F\times \mathcal{V})$
for each $T\in \mbox{TR}$. Therefore,
$\mbox{TR}\subseteq f^{-1}(\diamondsuit P_n^{-1}(A_F\times \mathcal{U}\times A_F\times \mathcal{V}))$.
Consequently,
$f^{-1}(\diamondsuit P_n^{-1}(A_F\times \mathcal{U}\times A_F\times \mathcal{V}))=\mbox{TR}$.

Case 2.1.2: $\prec_{\overline{T}}$ does not belong to the transitive closure of $\prec_{\overline{T}}^1,\prec_{\overline{T}}^2,\prec_{\overline{T}}^3$. That is, $\prec_{\overline{T}}$ must be generated by $\prec_{\overline{T}}^4$ as a member. By inspecting the above four relations, we obtain that, for each $T\in \mbox{TR}$, $\tau\prec_T\sigma$ ($\prec_{T}$ must be generated by $\prec_{T}^4$ as a member) if and only if
\begin{enumerate}
	\item $\tau=\underline{s}$ for some $s\in \omega^{\textless \omega}$;
	\item $\sigma\in \omega^{\textless \omega}$;
	\item $s$ is a strict prefix of $\sigma$;
	\item there exists $k\in \omega$ such that $|s|+1\leq k\leq |\sigma|$ and $\sigma|_k\not\in T$.
\end{enumerate}

Note that $\tau\prec_{\overline{T}}\sigma$ in this case. Then $\tau=\underline{s}$ for some $s\in \omega^{\textless \omega}$, $\sigma\in \omega^{\textless \omega}$, and $s$ is a strict prefix of $\sigma$.

Case 2.1.2.1: If $2^{\Omega}\not\in \mathcal{U}$ or $2^{\Omega}\not\in \mathcal{V}$, then we claim that $f^{-1}(\diamondsuit P_n^{-1}(A_F\times \mathcal{U}\times A_F\times \mathcal{V}))=\bigcup_{|s|+1\leq k\leq |\sigma|}U_{\sigma|_k}^c$.

For any $T\in f^{-1}(\diamondsuit P_n^{-1}(A_F\times \mathcal{U}\times A_F\times \mathcal{V}))$, there exists $I\in \mbox{RI}(\prec_T)$ such that $P_n(\eta(I))=(I,U_n^T,I,A_n^T)\in A_F\times \mathcal{U}\times A_F\times \mathcal{V}$. Assume that $\tau\not\prec_T \sigma$. Then $U_n^T=U_{h_1(\sigma,\tau)}=2^{\Omega}$ and $A_n^T=U_{h_1(\sigma,\tau)}=2^{\Omega}$. It follows that $U_n^T=2^{\Omega}\in \mathcal{U}$ and $A_n^T=2^{\Omega}\in \mathcal{V}$, which contradicts the assumption that $2^{\Omega}\not\in \mathcal{U}$ or $2^{\Omega}\not\in \mathcal{V}$. Hence, $\tau\prec_T \sigma$ for each $T\in f^{-1}(\diamondsuit P_n^{-1}(A_F\times \mathcal{U}\times A_F\times \mathcal{V}))$. Since $\tau=\underline{s}$ for some $s\in \omega^{\textless \omega}$ and $\sigma\in \omega^{\textless \omega}$, we can conclude that $\tau\prec_T \sigma$ must be generated by $\prec_T^4$ as a member. This means that $T\in \bigcup_{|s|+1\leq k\leq |\sigma|}U_{\sigma|_k}^c$.

Conversely, let $T\in \bigcup_{|s|+1\leq k\leq |\sigma|}U_{\sigma|_k}^c$. Then there exists $k\in \omega$ with $|s|+1\leq k\leq |\sigma|$ such that $T\in U_{\sigma|_k}^c$, that is, $\sigma|_k\not\in T$. This means that $\tau=\underline{s}\prec_T^4 \sigma|_{|s|+1}=\sigma|_k$ $(k=|s|+1)$, or
$\tau=\underline{s}\prec_T^2 \underline{\sigma|_{|s|+1}}\prec_T^2 \cdots \prec_T^2\underline{\sigma|_{k-1}}\prec_T^4\sigma|_k$
$(k\geq |s|+2)$. It follows that $\tau\prec_T \sigma$. Consequently, $U_n^T=U^T_{h_1(\sigma,\tau)}=U_s$ and $A_n^T=A^T_{h_1(\sigma,\tau)}=U_{\sigma}^c$.

Since $\tau\prec_{\overline{T}}\sigma$ and $\overline{T}\in f^{-1}(\diamondsuit P_n^{-1}(A_F\times \mathcal{U}\times A_F\times \mathcal{V}))$, there exists $I\in \mbox{RI}(\prec_{\overline{T}})$ such that $(I,U_n^{\overline{T}},I,A_n^{\overline{T}})=(I,U_\tau,I,U_{\sigma}^c)\in A_F\times \mathcal{U}\times A_F\times \mathcal{V}$. Note that $I=I_x$ or $J_x$ for some $x\in \omega^{\omega}$ and $J_x\subseteq I_x$. Hence, we have that for each $T\in \bigcup_{|s|+1\leq k\leq |\sigma|}U_{\sigma|_k}^c$, $(I_x,U_n^T,I_x,A_n^T)=(I_x,U_\tau,I_x,U_{\sigma}^c)\in A_F\times \mathcal{U}\times A_F\times \mathcal{V}$. It follows that
$(I_x,U_n^T,I_x,A_n^T)=(I_x,U_\tau,I_x,U_{\sigma}^c)=P_n(\eta(I_x))\in P_n(\eta(\mbox{RI}(\prec_T)))\cap (A_F\times \mathcal{U}\times A_F\times \mathcal{V})$
for each $T\in \bigcup_{|s|+1\leq k\leq |\sigma|}U_{\sigma|_k}^c$. Therefore,
$\bigcup_{|s|+1\leq k\leq |\sigma|}U_{\sigma|_k}^c\subseteq f^{-1}(\diamondsuit P_n^{-1}(A_F\times \mathcal{U}\times A_F\times \mathcal{V}))$.

We can thus conclude that
$f^{-1}(\diamondsuit P_n^{-1}(A_F\times \mathcal{U}\times A_F\times \mathcal{V}))=\bigcup_{|s|+1\leq k\leq |\sigma|}U_{\sigma|_k}^c$.

Case 2.1.2.2: $2^{\Omega}\in \mathcal{U}$ and $2^{\Omega}\in \mathcal{V}$. We claim that $f^{-1}(\diamondsuit P_n^{-1}(A_F\times \mathcal{U}\times A_F\times \mathcal{V}))=\mbox{TR}$.

Since $\tau\prec_{\overline{T}}\sigma$ and $\overline{T}\in f^{-1}(\diamondsuit P_n^{-1}(A_F\times \mathcal{U}\times A_F\times \mathcal{V}))$, there exists $I\in \mbox{RI}(\prec_{\overline{T}})$ such that $(I,U_n^{\overline{T}},I,A_n^{\overline{T}})=(I,U_\tau,I,U_{\sigma}^c)\in A_F\times \mathcal{U}\times A_F\times \mathcal{V}$. This means that $U_\tau\in \mathcal{U}$ and $U_{\sigma}^c\in \mathcal{V}$.

Note that $I=I_x$ or $I=J_x$ for some $x\in \omega^{\omega}$ and $J_x\subseteq I_x$. Then, for each $T\in \mbox{TR}$, $I_x\in \mbox{RI}(\prec_T)$ and $P_n(\eta(I_x))=(I_x,U_n^T,I_x,A_n^T)$. Since $U_n^T=U_\tau,A_n^T=U_{\sigma}^c$ or $U_n^T=2^{\Omega}, A_n^T=2^{\Omega}$, we have $U_n^T\in \mathcal{U}$ and $A_n^T\in \mathcal{V}$ in both cases, by the assumption that $2^{\Omega}\in \mathcal{U}$ and $2^{\Omega}\in \mathcal{V}$. It follows that $(I_x,U_n^T,I_x,A_n^T)\in A_F\times \mathcal{U}\times A_F\times \mathcal{V}$. Therefore, $P_n(\eta(I_x))\in P_n(\eta(\mbox{RI}(\prec_T)))\cap (A_F\times \mathcal{U}\times A_F\times \mathcal{V})$. Thus, $f^{-1}(\diamondsuit P_n^{-1}(A_F\times \mathcal{U}\times A_F\times \mathcal{V}))=\mbox{TR}$.

Case 2.2: For each $\overline{T}\in f^{-1}(\diamondsuit P_n^{-1}(A_F\times \mathcal{U}\times A_F\times \mathcal{V}))$, we have $\tau\not\prec_{\overline{T}}\sigma$. Then we claim that
$f^{-1}(\diamondsuit P_n^{-1}(A_F\times \mathcal{U}\times A_F\times \mathcal{V}))=\{T\in \mbox{TR}:\tau\not\prec_T \sigma\}=\{T\in \mbox{TR}:\tau\prec_T\sigma\}^c$.

By assumption, the inclusion from the left-hand side to the right-hand side is immediate. For the converse, let $T\in \mbox{TR}$ be such that $\tau\not\prec_T \sigma$. Pick $\overline{T}\in f^{-1}(\diamondsuit P_n^{-1}(A_F\times \mathcal{U}\times A_F\times \mathcal{V}))$. Then we can find $I\in \mbox{RI}(\prec_{\overline{T}})$ such that $P_n(\eta(I))=(I,U_n^{\overline{T}},I,A_n^{\overline{T}})=(I,2^{\Omega},I,2^{\Omega})\in A_F\times \mathcal{U}\times A_F\times \mathcal{V}$. This yields that $2^{\Omega}\in \mathcal{U}$ and $2^{\Omega}\in \mathcal{V}$.

Note that $I=I_x$ or $J_x$ for some $x\in \omega^{\omega}$ and $J_x\subseteq I_x$. Then, for each $T\in \mbox{TR}$ with $\tau\not\prec_T \sigma$, we have
$(I_x,U_n^T,I_x,A_n^T)=(I_x,2^{\Omega},I_x,2^{\Omega})=P_n(\eta(I_x))\in P_n(\eta(\mbox{RI}(\prec_T)))\cap (A_F\times \mathcal{U}\times A_F\times \mathcal{V}))$.
Therefore, we can conclude that
$f^{-1}(\diamondsuit P_n^{-1}(A_F\times \mathcal{U}\times A_F\times \mathcal{V}))=\{T\in \mbox{TR}:\tau\prec_T\sigma\}^c$.

It remains to show that $\{T\in \mbox{TR}:\tau\prec_T\sigma\}$ is Borel.

Case 2.2.1: $\tau\in \omega^{\textless \omega}$ and $\sigma\in \omega^{\textless\omega}$. Pick $\overline{T}\in f^{-1}(\diamondsuit P_n^{-1}(A_F\times \mathcal{U}\times A_F\times \mathcal{V}))$. Then $\tau\not\prec_{\overline{T}}\sigma$. This means that $\tau$ is not a strict prefix of $\sigma$. It follows that $\{T\in \mbox{TR}:\tau\prec_T\sigma\}$ is empty, and hence is Borel.

Case 2.2.2: $\tau\in \omega^{\textless \omega}$ and $\sigma=\underline{s}$, where $s\in \omega^{\textless\omega}$. Since $\tau\not\prec_{\overline{T}}\sigma$, $\tau$ is not a prefix of $s$. It follows that $\{T\in \mbox{TR}:\tau\prec_T\sigma\}$ is empty, and thus is Borel.

Case 2.2.3: $\tau=\underline{t}$, where $t\in \omega^{\textless\omega}$, and $\sigma\in \omega^{\textless\omega}$. If $t$ is not a strict prefix of $\sigma$, then $\{T\in \mbox{TR}:\tau\prec_T\sigma\}$ is empty, and hence, it is Borel. Else, $t$ is a strict prefix of $\sigma$.
By the proof of Case 2.1.2, we have
$\{T\in \mbox{TR}:\tau\prec_{T}\sigma\}=\bigcup_{|t|+1\leq k\leq |\sigma|}U_{\sigma|_k}^c$,
which is Borel.

Case 2.2.4: $\tau=\underline{t}$, where $t\in \omega^{\textless\omega}$, and $\sigma=\underline{s}$, where $s\in \omega^{\textless\omega}$. Since $\tau\not\prec_{\overline{T}}\sigma$, $t$ is not a strict prefix of $s$. Therefore, it follows immediately that $\{T\in \mbox{TR}:\tau\prec_{T}\sigma\}$ is empty.

We can now conclude that $\{T\in \mbox{TR}:\tau\prec_{T}\sigma\}$ is Borel. Hence, $f^{-1}(\diamondsuit P_n^{-1}(A_F\times \mathcal{U}\times A_F\times \mathcal{V}))$ is Borel.

\textbf{Case 3}: $n\in E$. Then there is $(\sigma,\tau)\in \Omega\times \Omega$ such that $h_2(\sigma,\tau)=n$. We only consider the nontrivial case where $f^{-1}(\diamondsuit P_n^{-1}(A_F\times \mathcal{U}\times A_F\times \mathcal{V}))$ is nonempty.

Note that $\mathcal{O}(2^{\Omega})$ is an $\omega$-algebraic domain, and $\big\{\ua\{F_i:i\leq n\}:n\in \omega, F_i\subseteq_{fin}\Omega\big\}$ are compact elements of $\mathcal{O}(2^{\Omega})$, where $\ua\{F_i:i\leq n\}=\{B\in 2^{\Omega}:\exists i\leq n\ F_i\subseteq B\}$. Therefore, $\big\{\ua(\ua\{F_i:i\leq n\}): n\in \omega, F_i\subseteq_{fin}\Omega\big\}$ forms a topological base for the Scott topology on $\mathcal{O}(2^{\Omega})$, where $\ua(\ua\{F_i:i\leq n\})=\{\mathcal{P}\in \mathcal{O}(2^{\Omega}): \ua\{F_i:i\leq n\}\subseteq \mathcal{P}\}$. Thus, without loss of generality, we may assume that $\mathcal{U}=\ua(\ua\{F_i:i\leq n\})$ for some $n\in \omega$, where $F_i\subseteq_{fin}\Omega$ for all $i\leq n$.

Pick $\overline{T}\in f^{-1}(\diamondsuit P_n^{-1}(A_F\times \mathcal{U}\times A_F\times \mathcal{V}))$. Then there exists $I\in \mbox{RI}(\prec_{\overline{T}})$ such that 
$$P_n(\eta(I))=(I,U_n^{\overline{T}},I,A_n^{\overline{T}})=\left(I,\bigcup_{\rho\in \ua_{\prec_{\overline{T}}}\sigma\cap \ua_{\prec_{\overline{T}}}\tau}U_\rho,I,(U_\sigma\cap U_{\tau})^c\right)\in A_F\times \mathcal{U}\times A_F\times \mathcal{V}. $$
This means that $(U_\sigma\cap U_{\tau})^c\in \mathcal{V}$ and $I\in A_F$.
We can also conclude that $I=I_x$ or $I=J_x$ for some $x\in \omega^{\omega}$ and so $J_x\subseteq I_x$.

Case 3.1: $\bigcup_{\rho\in \ua_{\prec_{\overline{T}}}\sigma\cap \ua_{\prec_{\overline{T}}}\tau}U_\rho=\varnothing$. Then $\varnothing\in \mathcal{U}$. By the Scott openness of $\mathcal{U}$, it follows that $\mathcal{U}=\mathcal{O}(2^{\Omega})$. We claim that $f^{-1}(\diamondsuit P_n^{-1}(A_F\times \mathcal{U}\times A_F\times \mathcal{V}))=\mbox{TR}$.

Note that $I=I_x$ or $I=J_x$ for some $x\in \omega^{\omega}$ and $J_x\subseteq I_x$. Then, for each $T\in \mbox{TR}$, we have $I_x\in \mbox{RI}(\prec_T)$ and
$P_n(\eta(I_x))=(I_x,U_n^T,I_x,A_n^T)=(I_x,\bigcup_{\rho\in \ua_{\prec_{T}}\sigma\cap \ua_{\prec_{T}}\tau}U_\rho, I_x, (U_\sigma\cap U_{\tau})^c)\in A_F\times \mathcal{O}(2^{\Omega})\times A_F\times \mathcal{V}=A_F\times \mathcal{U}\times A_F\times \mathcal{V}$.
Therefore, $f^{-1}(\diamondsuit P_n^{-1}(A_F\times \mathcal{U}\times A_F\times \mathcal{V}))=\mbox{TR}$.

Case 3.2: $\bigcup_{\rho\in \ua_{\prec_{\overline{T}}}\sigma\cap \ua_{\prec_{\overline{T}}}\tau}U_\rho\neq \varnothing$. Then $\ua_{\prec_{\overline{T}}}\sigma\cap \ua_{\prec_{\overline{T}}}\tau\neq \varnothing$. This means that $\sigma$ and $\tau$ have a common upper bound under $\prec_{\overline{T}}$. By inspecting the definition of $\prec_{\overline{T}}$, we can conclude that $\tau\prec_{\overline{T}}\sigma$, $\sigma\prec_{\overline{T}}\tau$, or $\sigma=\tau$. In what follows, we only consider the nontrivial case where $\varnothing\not\in \mathcal{U}$.

Note that for each $T\in \mbox{TR}$,
$$\begin{array}{rcl} & & (I_x,\bigcup_{\rho\in \ua_{\prec_{T}}\sigma\cap \ua_{\prec_{T}}\tau}U_\rho, I_x, (U_\sigma\cap U_{\tau})^c)\in A_F\times \mathcal{U}\times A_F\times \mathcal{V} \\
&\Longleftrightarrow & \bigcup_{\rho\in \ua_{\prec_{T}}\sigma\cap \ua_{\prec_{T}}\tau}U_\rho\in \mathcal{U}=\ua(\ua\{F_i:i\leq m\}) \\
&\Longleftrightarrow & \ua\{F_i:i\leq m\}\subseteq \bigcup_{\rho\in \ua_{\prec_{T}}\sigma\cap \ua_{\prec_{T}}\tau}U_\rho \\
&\Longleftrightarrow &  \forall i\leq m, F_i\in \bigcup_{\rho\in \ua_{\prec_{T}}\sigma\cap \ua_{\prec_{T}}\tau}U_\rho \\
& \Longleftrightarrow & \forall i\leq m, \exists \rho_i\in \ua_{\prec_{T}}\sigma\cap \ua_{\prec_{T}}\tau, F_i\in U_{\rho_i} \\
&\Longleftrightarrow & \forall i\leq m, \exists \rho_i\in \ua_{\prec_{T}}\sigma\cap \ua_{\prec_{T}}\tau, \rho_i\in F_i
\end{array}
$$

Case 3.2.1: $\sup\{\sigma,\tau\}=s\in \omega^{\textless\omega}$. Then $\ua_{\prec_{T}}\sigma\cap \ua_{\prec_{T}}\tau=\{t,\underline{t}:t\in \omega^{\textless\omega}, s~\mathrm{is~a~strict~prefix~of}~t\}\cup \{\underline{s}\}$. 
We claim that $f^{-1}(\diamondsuit P_n^{-1}(A_F\times \mathcal{U}\times A_F\times \mathcal{V}))=\mbox{TR}$.

We have picked $\overline{T}\in f^{-1}(\diamondsuit P_n^{-1}(A_F\times \mathcal{U}\times A_F\times \mathcal{V}))$ and $I_x\in \mbox{RI}(\prec_{\overline{T}})$ such that $P_n(\eta(I_x))=(I_x,U_n^{\overline{T}},I_x,A_n^{\overline{T}})=(I_x,\bigcup_{\rho\in \ua_{\prec_{T}}\sigma\cap \ua_{\prec_{T}}\tau}U_\rho,I_x,(U_\sigma\cap U_{\tau})^c)\in A_F\times \mathcal{U}\times A_F\times \mathcal{V}$.

Because $s\in \omega^{\textless\omega}$, $s\prec_{\overline{T}} \rho_i$ must belong to the transitive closure of $\prec_{\overline{T}}^1,\prec_{\overline{T}}^2$, and $\prec_{\overline{T}}^3$. Note that, for any $T\in \mbox{TR}$, $\prec_T^j$ are the same as $\prec_{\overline{T}}^j$ for $j\in\{1,2,3\}$. This means that $\forall i\leq m, \exists \rho_i\in \ua_{\prec_{T}}\sigma\cap \ua_{\prec_{T}}\tau$ such that $\rho_i\in F_i$. Note that $m,\sigma,\tau, F_i$ are all fixed. Therefore, for each $T\in \mbox{TR}$,
$P_n(\eta(I_x))=(I_x,U_n^T,I_x,A_n^T)=(I_x,\bigcup_{\rho\in \ua_{\prec_{T}}\sigma\cap \ua_{\prec_{T}}\tau}U_\rho, I_x, (U_\sigma\cap U_{\tau})^c)\in A_F\times \mathcal{U}\times A_F\times \mathcal{V}$. It follows that $\mbox{TR}\subseteq f^{-1}(\diamondsuit P_n^{-1}(A_F\times \mathcal{U}\times A_F\times \mathcal{V}))$. As a result, $f^{-1}(\diamondsuit P_n^{-1}(A_F\times \mathcal{U}\times A_F\times \mathcal{V}))=\mbox{TR}$.

Case 3.2.2: $\sup\{\sigma,\tau\}=\underline{s}$, where $s\in \omega^{\textless\omega}$. 
Then 
$$\ua_{\prec_{T}}\sigma\cap \ua_{\prec_{T}}\tau=\{\underline{t}:t\in \omega^{\textless\omega}, s~\mathrm{is~a~strict~prefix~of}~t\}\cup \{u\in \omega^{\textless\omega}:\exists v\in B,~ v~\mathrm{is~a~prefix~of}~u\},$$ where $B=\{v\in \omega^{\textless\omega}: v\not\in T, s~\mathrm{is~a~strict~prefix~of}~v\}$.

Case 3.2.2.1: For all $i\leq m$, $F_i\cap \{\underline{t}:s~\mathrm{is~a~strict~prefix~of}~t\}\neq \varnothing$. Then we choose $\underline{t_i}\in F_i$ such that $s$ is a strict prefix of $t_i$ for each $i\leq m$. Thus $\underline{s}\prec_T \underline{t_i}$ for each $T\in \mbox{TR}$. This means that $\forall i\leq m, \exists \underline{t_i}\in \ua_{\prec_{T}}\sigma\cap \ua_{\prec_{T}}\tau$ such that $\underline{t_i}\in F_i$ for each $T\in \mbox{TR}$. Therefore, for each $T\in \mbox{TR}$, we find $I_x\in \mbox{RI}(\prec_T)$ such that $P_n(\eta(I_x))=(I_x,U_n^T,I_x,A_n^T)=(I_x,\bigcup_{\rho\in \ua_{\prec_{T}}\sigma\cap \ua_{\prec_{T}}\tau}U_\rho, I_x, (U_\sigma\cap U_{\tau})^c)\in A_F\times \mathcal{U}\times A_F\times \mathcal{V}$. Consequently, $f^{-1}(\diamondsuit P_n^{-1}(A_F\times \mathcal{U}\times A_F\times \mathcal{V}))=\mbox{TR}$.

Case 3.2.2.2: There exists $i\leq m$ such that $F_i\cap\{\underline{t}:s~\mathrm{is~a~strict~prefix~of}~t\}= \varnothing$. Let $J=\big\{i\leq m: F_i\cap \{\underline{t}:s~\mathrm{is~a~strict~prefix~of}~t\}= \varnothing\big\}$. We claim that 
$$f^{-1}(\diamondsuit P_n^{-1}(A_F\times \mathcal{U}\times A_F\times \mathcal{V}))=\bigcap_{i\in J}\bigcup_{v\in A_i}U_v^c,$$ where $A_i=\{v\in B:\exists u\in F_i, v~\mathrm{is~a~prefix~of}~u\}$ for each $i\in J$.

In fact, for each $i\in J$, $F_i\cap \{\underline{t}:s~\mathrm{is~a~strict~prefix~of}~t\}= \varnothing$ by the definition of $J$. It follows that $F_i\cap \{u\in \omega^{\textless\omega}:\exists v\in B,~ v~\mathrm{is~a~prefix~of}~u\}\neq \varnothing$ for each $i\in J$. This yields that there exist $u\in \omega^{\textless\omega}\cap F_i$ and $v\in B$ such that $v$ is a prefix of $u$. We can observe that $v\in A_i$ and $T\in U_v^c$ by the construction of $B$. Hence, $f^{-1}(\diamondsuit P_n^{-1}(A_F\times \mathcal{U}\times A_F\times \mathcal{V}))\subseteq\bigcap_{i\in J}\bigcup_{v\in A_i}U_v^c$.

Conversely, let $T\in \bigcap_{i\in J}\bigcup_{v\in A_i}U_v^c$. Then, for each $i\in J$, there exists $v_i\in A_i$ such that $v_i\not\in T$. Since $v_i\in A_i$, there exists $u_i\in F_i$ such that $v_i~\mathrm{is~a~prefix~of}~u_i$, and $v_i\in B$. By the construction of $B$, we have $v_i\in \ua_{\prec_{T}}\sigma\cap \ua_{\prec_{T}}\tau$. This implies that $u_i\in \ua_{\prec_{T}}\sigma\cap \ua_{\prec_{T}}\tau\cap F_i$. Therefore, for all $i\in J$, $\exists u_i\in \ua_{\prec_{T}}\sigma\cap \ua_{\prec_{T}}\tau, u_i\in F_i$.

For each $i\not\in J$, we have $F_i\cap \{\underline{t}:s~\mathrm{is~a~strict~prefix~of}~t\}\neq \varnothing$. By a similar argument to that in Case 3.2.2.1, we obtain that $\forall i\not\in J, \exists \rho_i\in \ua_{\prec_{T}}\sigma\cap \ua_{\prec_{T}}\tau, \rho_i\in F_i$ for each $T\in \mbox{TR}$. In conclusion, $\forall i\leq m, \exists \rho_i\in \ua_{\prec_{T}}\sigma\cap \ua_{\prec_{T}}\tau, \rho_i\in F_i$ for each $T\in \bigcap_{i\in J}\bigcup_{v\in A_i}U_v^c$. This entails that $\bigcap_{i\in J}\bigcup_{v\in A_i}U_v^c\subseteq f^{-1}(\diamondsuit P_n^{-1}(A_F\times \mathcal{U}\times A_F\times \mathcal{V}))$.

Thus, $f^{-1}(\diamondsuit P_n^{-1}(A_F\times \mathcal{U}\times A_F\times \mathcal{V}))=\bigcap_{i\in J}\bigcup_{v\in A_i}U_v^c$, which is Borel.

We conclude that $f$ is Borel. This completes the proof of Theorem~\ref{Pi}.

\section*{Acknowledgements}
{We thank Matthew de Brecht for useful discussions on the topic of this paper. This work is supported by the National Natural Science Foundation of China (No. 12231007, No. 12271263 and No. 12401596)}.

\bibliographystyle{plain}
\bibliography{model}

\end{document}